\documentclass{amsart}
\usepackage{amssymb,amsmath,amsthm,graphics,amscd,amsfonts}

\usepackage{color}

\theoremstyle{plain}
\newtheorem{theorem}{Theorem}[section]
\newtheorem{proposition}[theorem]{Proposition}
\newtheorem{corollary}[theorem]{Corollary}
\newtheorem{lemma}[theorem]{Lemma}

\newtheorem{remark}[theorem]{Remark}

\newcommand{\bfC}{{\mathbb C}}
\newcommand{\bfP}{{\mathbb P}}
\newcommand{\bfR}{{\mathbb R}}

\newcommand{\mapright}[1]{\smash{\mathop{   \hbox to 0.7cm{\rightarrowfill}}
  \limits^{#1}}}

\begin{document}

\title{Multiplier ideal sheaves and integral invariants 
on toric Fano manifolds}
\author{Akito Futaki}
\address{Department of Mathematics, Tokyo Institute of Technology, 2-12-1,
O-okayama, Meguro, Tokyo 152-8551, Japan}
\email{futaki@math.titech.ac.jp}
\author{Yuji Sano}
\address{Institut des Hautes \'Etudes Scientifiques \\
Le Bois-Marie, 35, route de Chatres\\
F-91440 Bures-sur-Yvette, France}
\email{sano@ihes.fr}

\date{March 8, 2008 }

\begin{abstract} 
%%revised_point
We extend Nadel's results on some conditions for the multiplier ideal sheaves to satisfy 
which are described in terms of an obstruction defined by the first author. Applying our
extension we can determine
the multiplier ideal subvarieties on toric del Pezzo surfaces which do not admit
K\"ahler-Einstein metrics. 
We also show
that one can define multiplier ideal sheaves for K\"ahler-Ricci solitons and extend the result of Nadel
using the holomorphic invariant defined by Tian and Zhu.
%%revised_point_end
\end{abstract}



\maketitle

\section{Introduction}
The existence problem of K\"ahler-Einstein metrics on compact complex manifolds of positive first Chern class, i.e. Fano manifolds, has not
been completely settled. The existence is conjectured to be equivalent to {\it K-stability}, see  \cite{donaldson02}.
There are existence results, for example by Nadel \cite{nadel1}, while there are known obstructions, for
example by the first author \cite{futaki83.1}. 

We recall Nadel's results on multiplier ideal sheaves (subschemes) in \cite{nadel1} and \cite{nadel2}.
In general the multiplier ideal sheaves are constructed as follows. Let $S=\{u_i\}$ be a sequence of K\"ahler potentials in $c_1(X)$ such that for any $\alpha \in (\frac{n}{n+1}, 1)$,
\begin{equation}
	\label{eq:S1}
		\sup u_i=0,
		\,\,\,
		\lim_{i\to \infty}\int_X \exp(-\alpha u_i)\omega^n_0=\infty,
\end{equation}
and for some non-empty open subset $U\subset X$
\begin{equation}\label{eq:bdd_condtion}
	\int_U \exp(-u_i)\omega^n_0 \le O(1).
\end{equation}
(Note that there always exists $U$ satisfying (\ref{eq:bdd_condtion}) due to the property of plurisubharmonic functions.) For $S$ as above, Nadel constructed a coherent ideal sheaf $\mathcal{I}_S\subset \mathcal{O}_X$, which is called the {\sl multiplier ideal sheaf}, satisfying
\begin{itemize}
	\item
		 $\mathcal{I}_S\neq \mathcal{O}_X$,
		 $\mathcal{I}_S\neq 0$
	\item
		For a non-negative Hermitian line bundle $L$,
		\begin{equation}\label{eq:vanishing_thm}
			H^q(X, L\otimes \mathcal{I}_S)=0
			\,\,\,\,
			(q>0).
		\end{equation}
\end{itemize}
Let us next recall the continuity method to find K\"ahler-Einstein metrics on Fano manifolds.
Let $(X, g)$ be an $n$-dimensional Fano manifold with a K\"ahler metric $g$ whose K\"ahler form $\omega_g$ represents $c_1(X)$. 
A K\"ahler metric $g$  is called a K\"ahler-Einstein (KE for short) metric if and only if 
\begin{equation}\label{eq:def_ke}
	\mbox{Ric}(\omega_{g})=\omega_{g}.
\end{equation}
Fix an initial K\"ahler metric $g$. Finding a solution of the equation (\ref{eq:def_ke}) is equivalent to finding $\varphi\in C^{\infty}(X)$ satisfying  the following Monge-Amp\`ere equation;
\begin{equation}\label{eq:ke_MA}
	\det(g_{i\bar{j}}+\varphi_{i\bar{j}})
	=
	\det(g_{i\bar{j}})
	\exp(h_g-\varphi)
\end{equation}
and $g_{i\bar{j}}+\varphi_{i\bar{j}}>0$, where
\begin{equation}\label{eq:einstein_gap}
	\mbox{Ric}(\omega_g)-\omega_g=\frac{\sqrt{-1}}{2\pi}\partial\bar{\partial}h_g,
	\,\,\,
	\int_X e^{h_g} \omega_g^n=\int_X \omega_g^n.
\end{equation}
To apply the continuity method to (\ref{eq:ke_MA}), we consider
\begin{equation}\label{eq:ke_CM}
	\det(g_{i\bar{j}}+(\varphi_{t})_{ i\bar{j}})
	=
	\det(g_{i\bar{j}})
	\exp(h_g-t\varphi_{t}),
\end{equation}
where $t\in [0,1]$.
We have a solution for $t=0$ by a theorem of Yau \cite{yau78}. By the implicit function theorem the subset $R$ consisting of all
$t$'s in $[0,1]$ for which (\ref{eq:ke_CM}) have a solution is an open subset. Therefore if $R$ is a closed subset then
$R = [0,1]$ and we have a solution for (\ref{eq:ke_MA}). 
If $R$ is not closed, i.e. if 
there is a sequence $\{t_{i}\}_i$ in $R$ such that $t_i \to t_{\infty}$ as $i \to \infty$ and that 
the equation (\ref{eq:ke_CM}) is not solvable at $t_{\infty}\in (0,1]$ then 
the sequence of solutions $S := \{\varphi_{t_i}\}$ of (\ref{eq:ke_CM}) 
satisfies (\ref{eq:S1}) and (\ref{eq:bdd_condtion}) and induces 
the multiplier ideal sheaf $\mathcal{I}_S\subset \mathcal{O}_X$.

We call the subscheme cut out by the multiplier ideal sheaf a {\it multiplier ideal subscheme}
and its support  $V$ the {\it multiplier ideal subvariety} (MIS for
short).
In this note, we call the MIS induced due to the bubble of the equation of (\ref{eq:ke_CM}) the {\it 
K\"ahler-Einstein multiplier ideal subvariety} or KE-MIS for short.
In \cite{nadel2}, Nadel compared the KE-MIS with the zero sets of global holomorphic vector fields for
which the obstruction defined by the first author (\cite{futaki83.1}) vanishes.
This obstruction is the map $F: \mathfrak h(X) \to \mathbb{C}$ defined as
\begin{equation}\label{F}
	F(v):=\int_X dh_g(v)\,\omega_g^n,
\end{equation}
where $\mathfrak h(X)$ is the Lie algebra of all 
global holomorphic vector fields on $X$. $F$ is independent of the
choice the K\"ahler form $\omega_g$ 
in $c_1(X)$, is a Lie algebra character and vanishes if $X$ admits a K\"ahler-Einstein
metric. 
%%%%%%%%%%%%
\begin{theorem}[Nadel \cite{nadel2}]\label{thm:nadel}
Let $X$ be a Fano manifold and $g$ an initial K\"ahler metric on $X$ whose K\"ahler form represents $c_{1}(X)$. Suppose that the closedness does not hold for the continuity method, so that we get a multiplier ideal subvariety $V\subset X$. Then for any global holomorphic vector field $v$ on $X$ with $F(v) = 0$, we have $V\not\subset Z^{+}(v)$.
\end{theorem}
\noindent
Here $Z(v) \subset X$ denotes the zero set of $v$ and 
\[
	Z^{+}(v):=
	\{
		p\in Z(v) \mid
		\mbox{Re}(\mbox{div}(v)(p))>0
	\},
\]
where $\mbox{div}(v)=(\mathcal{L}_{v}vol_g)/vol_g$, $vol_g$ is the volume form induced by $g$ and $\mathcal{L}_v$ is the Lie derivative along $v$.
In general $\mbox{div}(v)$ will depend on our choice of volume form. However it is easy to check that at points where $v$ vanishes, $\mbox{div}(v)$ is well-defined and not depend on our choice of the volume form. Therefore, $Z^{+}(v)$ is a well-defined set. 

The purpose of this paper is to extend this result in two ways. Firstly, we give a refinement of Theorem \ref{thm:nadel}
for toric Fano manifolds by extending it to all holomorphic vector fields which are generated by the elements of 
the Lie algebra of the torus. Secondly, we show that if we consider the continuity method to
prove the existence of K\"ahler-Ricci soliton on Fano manifolds we obtain multiplier ideal sheaves
and subschemes, the supports of which
we call the {\it K\"ahler-Ricci soliton multiplier ideal subvarieties} or KRS-MIS for short. In this case 
we can extend Nadel's result using the holomorphic invariant defined by Tian and Zhu \cite{tian-zhu}.

To explain the first extension, let $X$ be a toric Fano manifold of complex dimension $n$. 
Then the $n$-dimensional torus $T \cong (\bfC^{\ast})^n$ 
acts on $X$ effectively. 
Let $T_{\bfR}$ be the real torus in $T$ and $\mathfrak t_{\bfR}$ be its Lie algebra.
Put $N_{\bfR} := J\mathfrak t_{\bfR}$ where $J$ is the complex structure of $T$.
For $\xi \in N_{\bfR}$, we denote by $v_{\xi}$ the holomorphic vector field induced by $\xi$,
i.e. if $\xi^{\sharp}$ is the real vector field induced by $\xi$ then 
$v_{\xi} = \frac12(\xi^{\sharp} - i(J\xi)^{\sharp})$.
If $X$ does
not admit K\"ahler-Einstein metrics 
there is a holomorphic vector field $v_{\xi}$ on $X$ induced by a vector $\xi \in N_{\bfR}$ such that the invariant $F(v_{\xi})$ is not equal to zero because Wang and Zhu (\cite{wang-zhu}) proved that a toric Fano manifold admits K\"ahler-Einstein metrics if and only if the obstruction $F$ vanishes. Note that
$F$ vanishes if and only if $F(v_{\xi}) = 0$ for any $\xi \in N_{\bfR}$, see \cite{mabuchi90} for the detail, but this follows also from the conclusions of \cite{wang-zhu}. Let $\Delta \subset M_{\bfR}$ be the reflexive polytope defining
the toric Fano manifold $X$ where $M_{\bfR}$ is the dual space of $N_{\bfR}$ (these are the standard
notations in toric geometry, see section 2).
For $\xi \in N_{\bfR}$, we define the subset of $\Delta$
\[
	D^{\le 0}(\xi):=\{y\in \Delta \mid \langle y, \xi \rangle \le 0\}.
\]
%%%%%%%
\begin{theorem}
	\label{thm:main_theorem}
	Let $X$ be a toric Fano manifold which does not admit K\"ahler-Einstein metrics. Let $G$ be a maximal compact subgroup of the automorphism group of $X$, $G^{\bfC}$ be its complexification
(so that $T$ is a maximal torus of $G^{\bfC}$), and $V$ be the $G^{\mathbb{C}}$-invariant KE-MIS induced by choosing the $G$-invariant metric $g$ as the initial metric in (\ref{eq:ke_CM}). Suppose that there is a vector $\xi\in N_{\bfR}$ such the invariant $F(v_{\xi})$  is positive.
Then $\mu_g(V) \not\subset D^{\le 0}(\xi)$.
\end{theorem}
\noindent
Note that $F(v_{\xi})$ is real because the imaginary part of $v_{\xi}$ is a Killing vector field.
%%revised_point
Using Theorem  \ref{thm:main_theorem} we determine the KE-MIS  of the surface obtained the blowing up $\mathbb{CP}^2$ at a single point which
do not admit K\"ahler-Einstein metrics.
%%revised_point_end

It is a natural expectation that there is a linkage of KE-MIS's and the GIT stability conditions.
One candidate of GIT stability is the slope stability introduced by Ross and Thomas \cite{ross-thomas}.
The slope stability is defined so as to check K-stability for subschemes and 
test configurations constructed using those subschemes in a natural way. 
But recently Panov and Ross \cite{panov-ross07}
proved that the surface obtained by blowing up $\bfC\bfP^2$ at two points
is slope stable with respect to the anticanonical polarization, 
and thus the KE-MIS does not destabilize this surface. 
Note that the exceptional divisor destabilizes the surface obtained by blowing up 
$\bfC\bfP^2$ at a single point, see Example 3.9 in \cite{panov-ross07}.
We summarize our results in the following.

\begin{corollary}\label{thm:main_theorem2}
%%revised_point
If $X$ is the surface obtained by blowing up $\bfC\bfP^2$ at one point, 
then there is a compact subgroup $G'$ of the automorphism group of $X$ such that the $(G')^{\bfC}$-invariant KE-MIS is the exceptional divisor.
This KE-MIS destabilizes with respect to the anticanonical polarization. 
\end{corollary}
%%revised_point_end
%%revised_point
\begin{remark}\label{rem:mis_flow}
After announcing the early version of this paper, the authors found that Theorem \ref{thm:main_theorem} is not enough to determine the KE-MIS on the surface obtained by blowing up $\bfC\bfP^2$ at two points $p_1$ and $p_2$. However, in the view point of the K\"ahler-Ricci flow, it is possible to determine the induced MIS (see Appendix).
More precisely, the MIS from the K\"ahler-Ricc flow on $\mathbb{CP}^2\sharp 2\overline{\mathbb{CP}^2}$ equals to the tree of all the three $(-1)$-curves.
\end{remark}
%%revised_point_end

We expect that there is some other formulation of slope stability of subschemes in which
the K\"ahler-Einstein multiplier ideal subschemes always destabilize the Fano manifolds.

To explain the second extension of Nadel's result, let us recall the definition of K\"ahler-Ricci soliton and its continuity method. Let $X$ be an $n$-dimensional Fano manifold, $K$ a
maximal compact subgroup of the holomorphic automorphism group of $X$,
$\mathfrak k(X)$ the Lie algebra of $K$ and $\mathfrak h_r(X)$ the 
Lie algebra of the complexification of $K$.
In other words $\mathfrak h_r(X)$ is the reductive part of the Lie algebra $\mathfrak h(X)$ 
of all holomorphic vector fields on $X$ and $\mathfrak k(X)$ is the compact 
real form of $\mathfrak h_r(X)$.
Let  $g$ be a $K$-invariant K\"ahler metric 
which represents $c_1(X)$ and a holomorphic vector field $v \in \mathfrak h_r(X)$. 
A pair $(g, v)$  is called a {\it K\"ahler-Ricci soliton} (KR soliton for short) if and only if 
\begin{equation}\label{eq:def_riccisoliton}
	\mbox{Ric}(\omega_{g})-\omega_{g}=L_{v}(\omega_{g})
\end{equation}
where $\omega_g$ is the K\"ahler form of $g$ and $L_{v}$ is the Lie derivative along $v$.  
Fix a $K$-invariant K\"ahler metric $g^0$ whose K\"ahler form is in $c_1(M)$. 
Finding a solution of the equation (\ref{eq:def_riccisoliton}) is equivalent to finding $\varphi\in C^{\infty}(M)$ satisfying  the following Monge-Amp\`ere equation;
\begin{equation}\label{eq:riccisoliton_MA}
	\det(g^0_{i\bar{j}}+\varphi_{i\bar{j}})
	=
	\det(g^0_{i\bar{j}})
	\exp\{h_0-\theta_{v,0}-v(\varphi)-\varphi\}
\end{equation}
and $(g_{i\bar{j}})=(g^0_{i\bar{j}}+\varphi_{i\bar{j}})>0$ where
\begin{equation}\label{eq:einstein_gapkrs}
	\mbox{Ric}(\omega_0)-\omega_0=\frac{\sqrt{-1}}{2\pi}\partial\bar{\partial}h_0,
	\,\,\,
	\int_M e^{h_0} \omega_0^n=\int_M \omega_0^n,	
\end{equation}
\begin{equation}\label{eq:def_theta}
	i_v\omega_0=\frac{\sqrt{-1}}{2\pi} \bar{\partial}\theta_{v,0},
	\,\,\,
	\int_M e^{\theta_{v,0}}\omega^n_0=\int_M \omega_0^n.
\end{equation}
To apply the continuity method to (\ref{eq:riccisoliton_MA}), we consider
\begin{equation}\label{eq:riccisoliton_CM}
	\det(g^0_{i\bar{j}}+(\varphi_{t})_{ i\bar{j}})
	=
	\det(g^0_{i\bar{j}})
	\exp\{h_0-\theta_{v,0}-v(\varphi_{t})-t\varphi_{t}\}
\end{equation}
where $t\in [0,1]$.
Let $R$ be the subset of $[0,1]$ consisting of all
$t$'s in $[0,1]$ for which (\ref{eq:riccisoliton_CM}) has a solution.
As is proved by Zhu in \cite{Zhu} there exists a solution for $t = 0$, and thus 
$R$ is a non-empty open subset. 
Therefore if $R$ is a closed subset then
$R = [0,1]$ and we have a solution for (\ref{eq:riccisoliton_MA}). 
We will show in section 4 that if $R$ is not closed, i.e. if 
there is a sequence $\{t_{i}\}_i$ in $R$ such that $t_i \to t_{\infty}$ as $i \to \infty$ and that 
the equation (\ref{eq:riccisoliton_CM}) is not solvable at $t_{\infty}\in (0,1]$ then 
the sequence of solutions $S := \{\varphi_{t_i}\}$ of (\ref{eq:riccisoliton_CM}) induces 
a coherent ideal sheaf $\mathcal{I}_S\subset \mathcal{O}_X$. We call $\mathcal{I}_S$  the {\it K\"ahler-Ricci soliton multiplier ideal sheaf} and the support of the corresponding subscheme
the {\it K\"ahler-Ricci soliton multiplier ideal subvariety} or KRS-MIS for short.

Now we consider the K\"ahler-Ricci soliton version of Theorem \ref{thm:nadel}. In the K\"ahler-Ricci 
soliton case, we shall use the
holomorphic invariant $F_v$ introduced by Tian and Zhu \cite{tian-zhu} instead of the invariant $F$.
Let us recall its definition. For a holomorphic vector field $v$, we define a holomorphic invariant $F_v:
\mathfrak h(X)\to \bfC$ by
\begin{equation}\label{eq:def_futakiinvariant}
	F_v(w)=\int_X w(h_g-\theta_{v,g})e^{\theta_{v,g}}\omega_g^n
\end{equation}
for $w\in \mathfrak h(X)$ where $h_g$ and $\theta_{v,g}$ are defined as in 
(\ref{eq:einstein_gapkrs}) and (\ref{eq:def_theta}) by using $\omega_g$ instead of $\omega_0$. 
Note that $F_v$ is independent of the choice of $g$ whose K\"ahler form belongs to $c_1(X)$ and that the vanishing condition $F_v(w)= 0$ for all $w$ is an obstruction to the existence of K\"ahler-Ricci solitons. From now on $v$ always denotes 
the holomorphic vector field in Lemma 2.2 of Tian and Zhu \cite{tian-zhu}. More precisely, Tian and Zhu proved that there is a unique holomorphic vector field $v\in \mathfrak h_r(X)$ such that $\mbox{Im}(v)\in 
\mathfrak k(X)$ and
\begin{equation}\label{eq:vanishing_tianzhu_inv}
	F_v(w)=0
	\,\,\,\,
	\mbox{for any }w\in \mathfrak h_r(X).
\end{equation}

\begin{theorem} \label{thm:nadel_thm_krs}
Let $(X, g)$ be a Fano manifold with a $K$-invariant K\"ahler metric $g$ whose K\"ahler form 
represents $c_1(X)$ and $v$ be as above. 
Suppose that closedness does not hold for the continuity method of (\ref{eq:riccisoliton_CM}) with respect to $v$, so that we get a multiplier ideal subvariety $V_v \subset X$. Then for any global holomorphic vector field $w \in \mathfrak h_r(X)$ we have $V_v \not\subset Z^{+}(w)$.
\end{theorem}

%%revised_point
As an application of Theorem \ref{thm:nadel_thm_krs} we prove that the surface obtained by blowing up $\mathbb{CP}^2$ at a single point
admit K\"ahler-Ricci solitons. Recall that Wang and Zhu \cite{wang-zhu}
proved the existence of K\"ahler-Ricci solitons on toric Fano
manifolds, so our proof is an alternate proof of their result.

Multiplier ideal sheaves can also be constructed from the K\"ahler-Ricci flow (\cite{PhSeSt06}, 
\cite{Rubin0708}, \cite{heier07-2}) 
and its discretization (\cite{Rubin0706}, \cite{Rubin0709}). 
As mentioned in Remark \ref{rem:mis_flow}, we shall discuss the multiplier ideal sheaves constructed from the K\"ahler-Ricci flow on the toric del Pezzo surfaces. More general arguments about them will be discussed somewhere else.
%%revised_point_end

\section{Proof of Theorem \ref{thm:main_theorem}}\label{sec:main_thm}
Throughout this section $X$ will be a toric Fano manifold.
Let us first recall necessary notations of toric geometry. 
Let $T$ be the $n$-dimensional algebraic torus $(\mathbb{C}^*)^n=\{(t_1,\dots,t_n)\mid t_i\in \mathbb{C}^*)\}$.
Denote by $M\simeq \mathbb{Z}^n$ the group of algebraic characters of $T$ and $N:=\mbox{Hom}(M,\mathbb{Z})$. For $y \in M \simeq \mathbb{Z}^n$ 
let $\chi^y\in\mbox{Hom}_{alg\,gp}(T, \mathbb{C}^*)$ be the character
\[
	\chi^y(t):=t_1^{y_1}t_2^{y_2}\cdots t_n^{y_n},
\]
where $t=(t_1,\dots, t_n)\in T$.
Let $M_{\mathbb{R}}:=M\otimes_{\mathbb{Z}}\mathbb{R}$ and $N_{\mathbb{R}}:=N\otimes_{\mathbb{Z}}\mathbb{R}$.
For $y=(y_1,\dots ,y_n)\in M_{\mathbb{R}}$ and $x=(x_1,\dots ,x_n)\in N_{\mathbb{R}}$, we define the canonical pairing $\langle y, x\rangle\in \mathbb{R}$ by
\[
	\langle y,x \rangle:=\sum_{i=1}^{n}y_ix_i.
\]
\noindent
Let $X=X(N, \Sigma_{\Delta})$ be a smooth toric Fano $n$-dimensional manifold defined by a reflexive polytope $\Delta\subset M_{\mathbb{R}}$. Here $\Sigma_{\Delta}$ is the corresponding fan.
Recall that for a complete fan $\Sigma$ determining a toric Fano manifold $X$, the reflexive polytope is expressed by 
\[
	\Delta=
	\{
		a\in M_{\mathbb{R}}
		\mid
		\langle a, b_{\rho}\rangle\le 1
		\,\,\,
		\mbox{for all one dimensional cone }
		\rho \subset \Sigma
	\},
\]
where $b_{\rho}$ is the primitive element of $\rho$.
For $t=(t_1,\dots,t_n)\in T$, we introduce affine logarithmic coordinates $x_i:=\log|t_i|^2,\,\,i=0,\dots,n$,  on $N_{\mathbb{R}}$ by regarding $T\simeq (\mathbb{C}^*)^n$ and $N_{\mathbb{R}}\simeq \mathbb{R}^n$.

We next consider the symmetries of $X$ (cf. \cite{batyrev-selivanova}).
For the maximal torus $T\subset \mbox{Aut}(X)$, let $\mathcal{N}(T)\subset \mbox{Aut}(X)$ be the normalizer of $T$. Since $T$ acting on $X$ has an open orbit $U\subset X$ then $\mathcal{N}(T)$ naturally acts on $U$. Let $\mathcal{W}(X):=\mathcal{N}(T)/T$ be the Weyl group. By choosing an arbitrary point $x_0\in U$ we can identity $U\simeq T$, and this identification gives a splitting of the following  short exact sequence
\[	
	1\to T \to \mathcal{N}(T) \to \mathcal{W}(X) \to 1,
\]
i.e., an embedding $\mathcal{W}(X)\hookrightarrow \mathcal{N}(T)$. Denote by $T_{\bfR}:=(S^1)^n$ the maximal compact subgroup in $T$. We choose $G$ to be the maximal compact subgroup in $\mathcal{N}(T)$ generated by $\mathcal{W}(X)$ and $T_{\bfR}$, so that we have the short exact sequence
\[
	1\to T_{\bfR} \to G \to \mathcal{W}(X) \to 1.
\]  
Note that the group $\mathcal{W}(X)$ is isomorphic to the finite group of all symmetries of $\Delta$ (resp. $\Sigma_{\Delta}$), i.e., $\mathcal{W}(X)$ is isomorphic to a subgroup of $\mbox{GL}(M)
\simeq \mbox{GL}(n,\mathbb{Z})$ (resp. $\mbox{GL}(N)$) consisting of all elements preserving $\Delta$ (resp. $\Sigma_{\Delta}$). (cf. Proposition 3.1 in \cite{batyrev-selivanova})
If we choose a $G$-invariant initial metric $g$ in (\ref{eq:ke_CM}), we find that the induced KE-MIS is $G^{\mathbb{C}}$-invariant, where $G^{\mathbb{C}}$ is the complexification of $G$, see the construction of MIS in \cite{nadel1}.

Let us introduce a $G$-invariant K\"ahler metric. Let $L(\Delta)=\{v_0,\dots,v_m\}:=M\cap \Delta$. 
Then $v_0,\dots,v_m$ determine algebraic characters $\chi_i: T\to \mathbb{C}^*$ of $T$, 
$i=0,\dots,m,$ such that
\[
	|\chi_i(t)|=e^{\langle v_i, x \rangle},
\] 
where $x$ is the image of $t$ under the canonical projection $T\to N_{\mathbb{R}}$.
Let us define the $\mathbb{R}$-valued function $u$ on the open dense orbit $U\simeq T$
\[
	u:=\log\Big( \sum_{i=0}^{m}|\chi_i(x)|\Big),
\]
which descends to a function on $N_{\mathbb{R}}$
\[
	\tilde{u}:=\log\Big( \sum_{i=0}^{m}e^{\langle v_i,x \rangle}\Big).
\]
We consider the $G$-invariant K\"ahler metric $g$ on $X$ such that the restriction of the corresponding differential $2$-form on $U$ is defined by 
\begin{equation}\label{eq:invariant_metric}
	\omega_{g}=\frac{\sqrt{-1}}{2\pi}\partial\bar{\partial} u.
\end{equation}
This metric is the pull-back of the Fubini-Study metric on $\mathbb{C}\bfP^m$ with respect to the anticanonical embedding $X\hookrightarrow \mathbb{C}\bfP^m$ defined by the algebraic characters $\chi_0, \dots, \chi_m$.
Let $\mu_g$ be a natural moment map with respect to $g$
\[
	\mu_g: X \to M_{\mathbb{R}},
\]
whose image $\mu_g(X)$ is $\Delta$, and
the $G$-equivariant moment map $\mu_{\tilde{u}}:N_{\mathbb{R}}\to M_{\mathbb{R}}$
defined by
\[
	\mu_{\tilde{u}}(x)
	=\mbox{grad}\,\tilde{u}(x):=
	\bigg(
		\frac{\partial \tilde{u}}{\partial x_1}(x),\dots, \frac{\partial \tilde{u}}{\partial x_n}(x)
	\bigg),
\]
where $x=(x_1,\dots,x_n).$
The latter is  a diffeomorphism of $N_{\mathbb{R}}$ onto the interior of $\Delta$.

\begin{proof}[Proof of Theorem \ref{thm:main_theorem}]
Our proof comes essentially from the proof of Theorem \ref{thm:nadel} by Nadel \cite{nadel2}. To make our proof self-contained,  let us go over Nadel's proof. Let $h_t\in C^{\infty}(X)$ be the function defined as in (\ref{eq:einstein_gap}) with respect to the evolved metric $g_t$ of (\ref{eq:ke_CM}). Then
\begin{eqnarray*}
		\frac{\sqrt{-1}}{2\pi}\partial\bar{\partial}h_t
	&=&
		\mbox{Ric}(\omega_t)-\omega_t
	=
		(t-1)(\omega_t-\omega_g)
	\\
	&=&
		(t-1)\frac{\sqrt{-1}}{2\pi}\partial\bar{\partial}\varphi_t,
\end{eqnarray*}
so $h_t$ is equal to $(t-1)\varphi_t$ up to a constant. Since (\ref{eq:ke_CM}) implies $\omega_t^n=e^{h_g-t\varphi_t}\omega_g^n$, so the invariant $F$ is given by
\[
	F(v)=(t-1)\int_X d\varphi_t (v)e^{h_g-t\varphi_t}\omega_g^n
\]
for any holomorphic vector field $v$. Since
\[
	F(v)
	=
	(t-1)
	\int_{X}
	(\mathcal{L}_{v}\varphi_t)
	e^{h_g-t\varphi_t}\omega_g^n
	=
	\frac{t-1}{-t}
	\int_{X}
	(\mathcal{L}_{v}e^{-t\varphi_t})
	e^{h_g}\omega_g^n,
\]
then we have
\[
	0=
	\frac{t}{t-1}F(v)+
	\int_{X}
	(\mathcal{L}_{v}e^{-t\varphi_t})
	e^{h_g}\omega_g^n.
\]
By the Leibniz product rule for Lie derivatives, we have
\[
	\mathcal{L}_{v}(e^{h_g-t\varphi_t}\omega_g^n)
	=
	(\mathcal{L}_{v}e^{-t\varphi_t})e^{h_g}\omega_g^n
	+
	e^{-t\varphi_t}\mathcal{L}_{v}(e^{h_g}\omega_g^n).
\]
Integrating over $X$, we get
\[
	0
	=
	\frac{t}{t-1}F(v)
	-
	\int_{X}e^{-t\varphi_t}\mathcal{L}_{v}(e^{h_g}\omega_g^n).
\]
Since we assumed that $X$ does not admit a K\"ahler-Einstein metric then 
 we have $t<1$. Introduce the notation $\mbox{div}_{1}(v)$ to denote the divergence of $v$ with respect to the (fixed, non-evolving) volume form $e^{h_g}\omega_g^n$. Then the integrand in the last integral can be rewritten as
\[
	e^{-t\varphi_t}\mbox{div}_{1}(v)e^{h_g}\omega_g^n=\mbox{div}_{1}(v)\omega_{t}^n.
\]
We conclude that
\begin{equation}\label{eq:futaki}
	0
	=
	\frac{t}{t-1}F(v)
	-
	\int_{X}\mbox{div}_{1}(v)\omega_t^n.	
\end{equation}

Now, let $\xi\in N_{\mathbb{R}}$ be the vector in the assumption of our main theorem and $v_{\xi}$ be the holomorphic vector field on $X$ induced by $\xi$. It is important to recall that $N_{\bfR}$ is 
identified with $J\mathfrak t_{\bfR}$, see section 1.
From (\ref{eq:invariant_metric}), on $U$
\[
	\omega_g^n
	=
	\bigg(\frac{1}{\pi}\bigg)^n
	\det (u_{ij}) dx_1\wedge \cdots \wedge dx_n \wedge d\Theta,
\]
where $d\Theta$ is the standard volume form of $T_{\bfR}$.
Since
\begin{eqnarray*}
		0
	&=&
		\partial\bar{\partial}\log(e^{h_g}\omega_g^n)
		+
		\bigg(\frac{\sqrt{-1}}{2\pi}\bigg)^{-1}\omega_g
	\\
	&=&
		\partial\bar{\partial}\log(e^{h_g+u}\det(u_{ij}))
\end{eqnarray*}
on $U$ and $\log(e^{h_g+u}\det(u_{ij}))$ is bounded (cf. Lemma 4.3 in \cite{song}), there is a constant $C$ such that
\begin{equation}\label{MA1}
	\log(e^{h_g}\det(u_{ij}))=-u+C
\end{equation}
on $U$.
Therefore
\begin{eqnarray}
	\nonumber
		\mbox{div}_{1}(v_{\xi})
	&=&
		\{\mathcal{L}_{v_{\xi}}(e^{h_g}\omega_g^n)\}/e^{h_g}\omega_g^n
	\\
	\label{eq:Lie_derivative_torus}
	&=&
		\mathcal{L}_{v_{\xi}}(\log(e^{h_g}\det(u_{ij})))
	\\
	\nonumber
	&=&
		-\mathcal{L}_{v_{\xi}}u
\end{eqnarray}
on $U$. Note that $\mathcal{L}_{v_{\xi}}u$ is real since the imaginary part of $v_{\xi}$ is a Killing vector
field and $u$ is $G$-invariant.
Note also that in (\ref{eq:Lie_derivative_torus}) the restriction $\mbox{div}_{1}(v_{\xi})|_U$ of the divergence of $v_{\xi}$ to $U$ descends to the function $-\frac{d}{ds}\bigg|_{s=0}\tilde{u}(x+s\xi)$
on $N_{\mathbb{R}}$, that is to say,
\[
		-\langle
			\mu_{\tilde{u}}(x), \xi
		\rangle
\]
for $x\in N_{\mathbb{R}}$. Hence, for $p\in X$ satisfying $\mu_{g}(p)\notin \partial \Delta$, we have
\begin{equation}\label{eq:divpositive_condition}
		\mu_{g}(p)\in D^{\le 0}(\xi)
		\iff
		\mbox{div}_{1}(v_{\xi})(p)\ge 0.
\end{equation}
Since $\mbox{div}_{1}(v_{\xi})$ is continuous on $X$, (\ref{eq:divpositive_condition}) still holds even if $\mu_{g}(p)\in \partial \Delta$. 
Assume that $\mu_g(V)$ would be contained in $D^{\le 0}(\xi)$. Then from (\ref{eq:divpositive_condition}), we have $\mbox{div}_{1}(v_{\xi})\ge 0$ on $V$. For sufficiently small $\varepsilon>0$, we define $W_{\varepsilon}:=\{p\in X \mid \mbox{div}_{1}(v_{\xi})(p)\le -\varepsilon\}$. Then $W_{\varepsilon}$ is a compact subset of $X-V$.
Here, let us recall that the volume of any compact subset of $X-V$ with respect to the evolved metric $\omega_{t_i}$ converges to zero as $i\to \infty$, where $\{t_i\}$ is the sequence such that $\varphi_{t_i}$ of solutions of (\ref{eq:ke_CM}) induces the KE-MIS (see Proposition 4.1 in \cite{nadel2}). So, we have
\[
		\int_{W_{\varepsilon}}\mbox{div}_{1}(v_{\xi})\omega_{t_i}^n \to 0.
\]
Then
\begin{eqnarray}
	\nonumber
		\int_X\mbox{div}_{1}(v_{\xi})\omega_{t_i}^n
	&=&
		\int_{X-W_{\varepsilon}}\mbox{div}_{1}(v_{\xi})\omega_{t_i}^n
		+\int_{W_{\varepsilon}}\mbox{div}_{1}(v_{\xi})\omega_{t_i}^n
	\\
	\label{eq:re_div_le_e}
	&\ge&
		-2\varepsilon\mbox{vol}(X),
\end{eqnarray}
for sufficiently large $i$, where $\mbox{vol}(X):=\int_{X}\omega_g^n$. Since $F(v_{\xi})>0$ and $\frac{t_i}{t_i-1}$ is negative and bounded from above, we have from (\ref{eq:re_div_le_e})
\[
	\frac{t_i}{t_i-1}F(v_{\xi})
	-\int_X\mbox{div}_{1}(v_{\xi})\omega_{t_i}^n
	<0
\]
for sufficiently large $i$. This contradicts (\ref{eq:futaki}). This completes the proof.
\end{proof}

\section{Proof of Corollary \ref{thm:main_theorem2}}
Let $X$ be the surface obtained by blowing up $\mathbb{CP}^2$ at $p_0=[1:0:0]$.
In this case, the set of vertices of the associated polytope $\Delta$ consists of
\begin{center}
	$w_1=(2, -1)$, $w_2=(-1, 2)$, $w_3=(-1, 0)$ and $w_4=(0, -1)$.
\end{center}
The set of the fundamental generators of one-dimensional cones in the associated fan $\Sigma$ consists of
\begin{center}
	$(1,1)$, $(-1,0)$, $(-1,-1)$ and $(0,-1)$.
\end{center}
We find that $X$ has the $\mathbb{Z}_2$-symmetry, whose set of fixed points in $M_{\mathbb{R}}$ equals to $\{(y, y)\in M_{\mathbb{R}}\mid y\in \mathbb{R}\}$.
The moment map
$
	\mu_{\tilde{u}}(x_1, x_2)
	=\mbox{grad}\tilde{u}
	:=
	\bigg(
		\frac{\partial \tilde{u}}{\partial x_1}(x), \frac{\partial \tilde{u}}{\partial x_2}(x)
	\bigg)
$ equals to
\begin{eqnarray}\label{moment1}
		\lefteqn{						
		e^{-\tilde{u}}
		\big(
		2e^{2x_1-x_2}+e^{x_1}-e^{-x_1+2x_2}-e^{-x_1+x_2}-e^{-x_1}+e^{x_1-x_2},
			}
		\hspace{1cm}
	\\
	&&
	\nonumber
		-e^{2x_1-x_2}+e^{x_2}+2e^{-x_1+2x_2}+e^{-x_1+x_2}-e^{-x_2}-e^{x_1-x_2}
	\big)	
\end{eqnarray}
where $e^{\tilde{u}}=e^{2x_1-x_2}+e^{x_1}+e^{x_2}+e^{-x_1+2x_2}+e^{-x_1+x_2}+e^{-x_1}+e^{-x_2}+e^{x_1-x_2}+1$.\\
In this section, we shall determine the KE-MIS on $X$  by using Theorem  \ref{thm:main_theorem}. 
For this purpose, we modify slightly $G$ defined in the previous section.
Let $p_1=[0:1:0]$ and $p_2=[0:0:1]$.
Let $E$ be the exceptional divisor of the above blow up.
Since the proper transform $\overline{p_0p_1}$ of the line passing $p_0$ and $p_1$ on the blow-up of $\bfC\bfP^2$ at $p_0$ has self-intersection zero, we can translate it in $X$.
In fact, 
\[
	\tau_1=
	\left(\begin{array}{ccc}
		1 & 0 & 0 
	\\
		0 & 1 & 0 
	\\
		0 & 1 & 1
	\end{array}\right)
\]
fixes $(1,0,0) \in \mathbb{C}^3$, so we find that $\tau_1 \in \mbox{Aut}(X)$ and the complexification $G_1^{\bfC}$ of $G_1:=\tau_1G\tau^{-1}_1$ gives rise to the continuous translations of the proper transform $\overline{p_0p_1}$ in $X$. By the same way, we have a compact group $G_2$ whose complexification $G^{\mathbb{C}}_2$ gives rise to the continuous translations of $\overline{p_0p_2}$.
Let $G'$ be the compact subgroup of $\mbox{Aut}(X)$ generated by $G$, $G_1$ and $G_2$.
From the invariance of MIS under these symmetry of $X$, we can reduce the possible MIS to the following two cases; the exceptional divisor $E$ or the $(+1)$-curve which does not intersect with $E$.
Let $v_{\xi}$ be the holomorphic vector field on $X$ induced by $\xi=(-1, -1)\in N_{\mathbb{R}}$. (The one-parameter subgroup induced by $v_{\xi}$ flows from the $(+1)$-curve towards the exceptional divisor.) By using the localization formula for the invariant $F$ (cf. \cite{futaki}, \cite{futakimorita85}), we find that the invariant $F(v_{\xi})$ is positive. 
This fact is computed also in the book \cite{tian2}. Since the segment $\overline{w_1w_2}$ in $M_{\mathbb{R}}$ between $w_1$ and $w_2$ which represents the $(+1)$-curve in $X$  is contained in $D^{\le 0}(\xi)$, so Theorem \ref{thm:main_theorem} implies that $\overline{w_1w_2}$ can not be the KE-MIS, i.e., the possibility of the $(+1)$-curve is ruled out.
Therefore the KE-MIS equals to $E$ exactly and the proof of  Corollary \ref{thm:main_theorem2} is completed.

\section{K\"ahler-Ricci soliton multiplier ideal subschemes}\label{sec:krs-mis}

In this section we show that if the solutions of the continuity method 
(\ref{eq:riccisoliton_CM}) for the existence of 
a K\"ahler-Ricci soliton do not converge as $t\to t_{\infty}$ then the blow up of the solutions implies the existence of multiplier ideal sheaves. Remark that we can construct a coherent ideal sheaf satisfying the same property as in the K\"ahler-Einstein case 
replacing the constant $\frac{n}{n+1}$ by any positive  constant $c<1$. 
For the purpose of this section it is sufficient to prove Proposition \ref{prop:tian_invariant_KRS} below, 
which corresponds to Tian's existence theorem of K\"ahler-Einstein metrics \cite{tian}.

We introduce a quantity defined by Mabuchi in \cite{mabuchi}, 
page 104, 
\[
	b:=\beta_v-\alpha_v>0,
\]
where
\[
	\alpha_v=\min_X\theta_{v,0},
	\,\,\,
	\beta_v=\max_X\theta_{v,0}.
\]
Note that in the case of K\"ahler-Ricci soliton, the function $\sigma(s)$ in \cite{mabuchi} is $-s+\mbox{constant}$ and that $\alpha_v$ and $\beta_v$ are independent of $g^0$.

\begin{proposition}\label{prop:tian_invariant_KRS}
	If there is a real constant $\alpha_0\in (\frac{n+b}{n+1+b},1)$
	and a uniform constant $C$ such that
	\begin{equation}\label{eq:tian_invariant_KRS}
		\int_X \exp(-\alpha_0(\varphi_t-\sup \varphi_t))e^{\theta_{v, 0}}\omega^n_0
		\le
		C
	\end{equation}
	for $t\in[0, t_{\infty})$,
	then (\ref{eq:riccisoliton_CM}) is solvable at $t=t_{\infty}$.
	In particular, if  (\ref{eq:riccisoliton_CM}) is not solvable at $t_{\infty}\in (0, 1]$, then
	there is a sequence $\{t_{i}\}$ such that $t_i \to t_{\infty}$ and
	\[
		\int_X \exp(-\alpha(\varphi_t-\sup \varphi_t))\omega^n_0
		\to
		\infty
	\]
	as $i\to \infty$ for any constant $\alpha \in (\frac{n+b}{n+1+b}, 1)$.
\end{proposition}

\begin{proof}
First, let us recall the functionals which are analogues to Aubin's functionals in the KE case (\cite{aubin84}). For a smooth function $\psi$ such that $(\mbox{Im}\,v)\psi = 0$ we put
\begin{eqnarray}
	\label{eq:fuctional_I}
		I_{v, g^0}(\psi)
	&=&
		\frac{1}{V_0}
		\int_X
			\psi(e^{\theta_{v,0}}\omega^n_0
			-e^{\theta_{v,0}+v(\psi)}\omega^n_{\psi}),
	\\
	\label{eq:functional_J}
		J_{v, g^0}(\psi)	
	&=&
		\frac{1}{V_0}
		\int^1_0 dt \int_X
			\dot{\psi}_t(e^{\theta_{v,0}}\omega^n_0-
			e^{\theta_{v,0}+v(\psi_t)}\omega^n_{\psi_t})
\end{eqnarray}
where $V_0 := \int_X \omega_0^n$. 
(For simplicity, we omit the subscripts and write $I(\psi)$ and $J(\psi)$ instead of $I_{v, g^0}(\psi)$ and $J_{v, g^0}(\psi)$.)
In the same manner as in the KE case, it is sufficient to estimate $J(\varphi_t)$ (or $I(\varphi_t)$)
from above uniformly.
Since
\[
	J(\varphi_t)-\frac{1}{V_0}\int_X \varphi_t e^{\theta_{v, 0}}\omega^n_0
	\le 0
\]
(cf. Lemma 5.2 in \cite{tian-zhu}), it is sufficient to estimate $\int_X \varphi_t e^{\theta_{v, 0}}\omega^n$ from above.
From (\ref{eq:riccisoliton_CM}) and (\ref{eq:tian_invariant_KRS}), we have
\[
	\int_X \exp((t-\alpha)\varphi_t-h_0+\theta_{v, 0}+v(\varphi_t))\omega^n_t
	\le
	C\exp(-\alpha \sup\varphi_t).
\]
From the concavity of the logarithm and $\int_X e^{\theta_{v, 0}+v(\varphi_t)}\omega^n_t=V_0$
(cf. \cite{tian-zhu00}, page 281) 
we have
\[
	\int_X ((t-\alpha)\varphi_t)e^{\theta_{v, 0}+v(\varphi_t)} \omega^n_t
	\le
	-\alpha \sup\varphi_t+C,
\]
so we get
\begin{equation}\label{eq:int_sup_int}
	\frac{1}{V_0}\int_X \varphi_t e^{\theta_{v, 0}}\omega^n
	\le
	\sup \varphi_t
	\le
	-\frac{t-\alpha}{\alpha}\frac{1}{V_0}\int_X \varphi_t e^{\theta_{v, 0}+v(\varphi_t)}\omega^n_t
	+C
\end{equation}
where the left inequality follows from (\ref{eq:def_theta}). 
From the calculations in \cite{tian-zhu}, page 319, we have 
\[
	\frac{d}{dt}
	\biggl(
		t(I(\varphi_t)-J(\varphi_t))
	\biggr)
	-(I(\varphi_t)-J(\varphi_t))
	=
	\frac{1}{V_0}\frac{d}{dt}
	\biggl(
		\int_X t(-\varphi_t)e^{\theta_{v, t}}\omega^n_t
	\biggr).
\]
So we have 
\begin{eqnarray}
	\nonumber
		\frac{1}{V_0}
		\biggl(
			\int_X \varphi_te^{\theta_{v, t}}\omega^n_t
		\biggr)
	&=&
		-(I(\varphi_t)-J(\varphi_t))
		+
		\frac{1}{t}\int^t_0
		(I(\varphi_s)-J(\varphi_s))
		ds
	\\
	\label{eq:negative_integral_varphi}
	&\le&
	0
\end{eqnarray}
because $\frac{d}{ds}(I(\varphi_s)-J(\varphi_s))\ge 0$ by Lemma 3.2
in \cite{tian-zhu00}.
Here $\theta_{v,t}$ is a fuction defined by
\[
	i_v\omega_t=\frac{\sqrt{-1}}{2\pi} \bar{\partial}\theta_{v,t},
	\,\,\,
	\int_X e^{\theta_{v,t}}\omega^n_t=\int_X \omega_t^n.	
\]
Note that $\theta_{v,t}=\theta_{v, 0}+v(\varphi_t)$ (cf. page 301 in \cite{tian-zhu}).
From (\ref{eq:int_sup_int}) and (\ref{eq:negative_integral_varphi}),
it is sufficient to consider the case when $t>\alpha$.
By the analogy of KE case, we have
\begin{eqnarray}
	\nonumber
		0
	&\ge&
		J(\varphi_t)-\frac{1}{V_0}\int_X \varphi_t e^{\theta_{v, 0}}\omega^n_0
		\,\,\,\,\,
		\mbox{(by Lemma 5.2 in \cite{tian-zhu})}
	\\
	\nonumber
	&\ge&
		\frac{1}{n+1+b}I(\varphi_t)
		-\frac{1}{V_0}\int_X \varphi_t e^{\theta_{v, 0}}\omega^n_0
		\,\,\,\,
		\mbox{(by Proposition A.1 in \cite{mabuchi})}
	\\
	\nonumber
	&=&
		-\frac{1}{n+1+b}
		\frac{1}{V_0}\int_X \varphi_t e^{\theta_{v, 0}+v(\varphi_t)}\omega^n_t
		-\frac{n+b}{n+1+b}
		\frac{1}{V_0}\int_X \varphi_t e^{\theta_{v, 0}}\omega^n_0
	\\
	\label{eq:estimate_int_varphi}
	&\ge&
		\frac{n+b}{n+1+b}
		\biggl(
			\frac{\alpha}{(t-\alpha)(n+b)}
			-1
		\biggr)
		\frac{1}{V_0}\int_X \varphi_t e^{\theta_{v, 0}}\omega^n_0
		-C.
		\,\,\,\,
		\mbox{(by (\ref{eq:int_sup_int}))}
\end{eqnarray}
From the assumption of this propostion, we have
\begin{eqnarray}\label{eq:estimate_coefficient}
	\nonumber
		\frac{\alpha}{(t-\alpha)(n+b)}
	&\ge&
		\frac{\alpha}{(1-\alpha)(n+b)}
	\\
	& > &
		\frac{n+b}{n+1+b}
		\cdot
		\frac{1}{n+b}
		\cdot
		\frac{n+1+b}{1}	
	=
		1.
\end{eqnarray}
From (\ref{eq:estimate_int_varphi}) and (\ref{eq:estimate_coefficient}) we get the desired uniform estimate of $\int_X \varphi_t e^{\theta_{v, 0}}\omega^n_0$. This completes the proof. 
\end{proof}

\section{Proof of Theorem \ref{thm:nadel_thm_krs}}

In this section we prove Theorem  \ref{thm:nadel_thm_krs}.

\begin{proof}[Proof of Theorem \ref{thm:nadel_thm_krs}]
This proof is almost the same as the proof of Theorem \ref{thm:nadel}.
For a given compact set $Y\subset X-V_v$, there is a constant $\gamma >\frac{n+b}{n+1+b}$ satisfying
\begin{equation}\label{eq:bdd_condtion2}
	\int_Y \exp (-\gamma (\varphi_t - \sup\varphi_t))\omega^n_0 < O(1).
\end{equation}
(See Section 2.12 in  \cite{nadel1}.) 
First, we prove

\begin{lemma}\label{lem:volume_zero}
\begin{equation}\label{eq:volume_zero}
	\int_Y e^{\theta_{v.t}}\omega^n_t \to 0
\end{equation}
as $t \to t_{\infty}$.
\end{lemma}

%%%%%
\begin{proof}[Proof of Lemma \ref{lem:volume_zero}]
From (\ref{eq:riccisoliton_CM})
\begin{eqnarray}
	\nonumber
		\int_Y e^{\theta_{v.t}}\omega^n_t 
	&=&
		\int_Y e^{h_0-t\varphi_t}\omega^n_0
		\le
		C(\sup e^{(\gamma -t)\varphi_t})\int_Y e^{-\gamma \varphi_t}\omega^n_0
	\\
	\nonumber
	&\le&
		C(\sup e^{(\gamma -t)\varphi_t})(\inf e^{-\gamma \varphi_t}).
\end{eqnarray}
The last inequality comes from (\ref{eq:bdd_condtion2}).
In fact, from (\ref{eq:bdd_condtion2}), we have
\[
	(\sup e^{\gamma\varphi_t})
	\int_Y e^{-\gamma\varphi_t}\omega^n_0 \le
	\mathcal{O}(1).
\]
Assume that the closedness of continuity method fails at $t_{\infty} \in (0,1]$. Then,
it is sufficient to show that $(\sup e^{(\gamma -t)\varphi_t})(\inf e^{-\gamma \varphi_t})\to 0$ as $t\to t_{\infty}$.
If $t_{\infty} \le \gamma$, we find
\begin{eqnarray*}
		(\sup e^{(\gamma -t)\varphi_t})(\inf e^{-\gamma \varphi_t})
	&=&
		e^{-t\sup \varphi_t}
	\\
	&\to&	
		0,
\end{eqnarray*}
since $\sup\varphi_t\to\infty$. So, it is sufficient to consider the case when $t_{\infty} >\gamma$. For 
$t>\gamma$, we find
\begin{eqnarray}
	\label{eq:t>gamma_1}
		(\sup e^{(\gamma -t)\varphi_t})(\inf e^{-\gamma \varphi_t})
	&=&
		e^{-\gamma \sup \varphi_t +(t-\gamma)\sup (-\varphi)}.
\end{eqnarray}
To estimate $\sup (-\varphi_t)$ by $\sup \varphi$, we need the following Harnack-type inequality as the case of K\"ahler-Einstein metrics; 
\begin{equation}\label{eq:harnack_krs}
	\sup (-\varphi_t) \le (n+b)  \sup \varphi_t +C.
\end{equation}
From now on, we shall prove (\ref{eq:harnack_krs}) in the case of K\"ahler-Ricci solitons. 
(The proof is essential same as the inequality (3.19) in \cite{wang-zhu}.)
From the calculations of Tian-Zhu (page 23 in \cite{tian-zhu}), for $t\in (\gamma, t_{\infty})$ we have
\begin{eqnarray}
		\frac{1}{V}
			\int_M -\varphi_te^{\theta_{X, t}}\omega^n_t
	&=&
		(I(\varphi_t)-J(\varphi_t))
		-
		\frac{1}{t}\int^t_0
		(I(\varphi_s)-J(\varphi_s))
		ds
	\nonumber \\
	&\le&
		(I(\varphi_t)-J(\varphi_t))
		-\frac{1}{t_0}(I(\varphi_0)-J(\varphi_0))  \label{eq:pf_harnack_krs_1}
	 \\	
	&\le&
		\frac{n+b}{n+b+1}I(\varphi_t)-C \nonumber
\end{eqnarray}
where 
the first inequality follows because 
 $I(\varphi_s)-J(\varphi_s) $
is positive and increasing in s	
and 
the last inequality follows from Proposition A.1 in \cite{mabuchi}.
Then we find
\begin{eqnarray}
		\frac{1}{V}
		\int_M -\varphi_te^{\theta_{X, t}}\omega^n_t
	&\le&
		(n+b)\frac{1}{V}
		\int_M \varphi_te^{\theta_{X, 0}}\omega^n_0
		-C \label{eq:pf_harnack_krs_2}
	\\
		&\le&
		(n+b)\sup\varphi_t-C. \nonumber
\end{eqnarray}
From Theorem 2.1 in \cite{cao-tian-zhu}, we find
\begin{equation}
	\label{eq:cao-tian-zhu}
		\sup (-\varphi_t)
		\le
		\frac{1}{V}
		\int_M -\varphi_te^{\theta_{X, t}}\omega^n_t+C.
\end{equation}
Combining (\ref{eq:pf_harnack_krs_2}) and (\ref{eq:cao-tian-zhu}),
we get (\ref{eq:harnack_krs}), proving the Harnack inequality.
From (\ref{eq:t>gamma_1}) and (\ref{eq:harnack_krs}), we find
\begin{eqnarray*}
		(\sup e^{(\gamma -t)\varphi_t})(\inf e^{-\gamma \varphi_t})
	&\le&
		\exp([(t-\gamma)(n+b)-\gamma]\sup \varphi_t)
	\\
	&=&
		\exp
		\biggl(
		\biggl[(n+b+1)
		\biggl(\frac{n+b}{n+b+1}-\gamma
		\biggr)\biggr]\sup\varphi_t \biggr)
	\\
	&\to& 0,
\end{eqnarray*}
since $\gamma>\frac{n+b}{n+b+1}$.
This completes the proof of Lemma \ref{lem:volume_zero}.
\end{proof}

Let us return to the proof of Theorem \ref{thm:nadel_thm_krs}. 
From (\ref{eq:riccisoliton_CM}) and the invariance of $F_v$, we have
\begin{eqnarray}
	\nonumber
		F_v(w)
	&=&
		\int_X d(h_{t}-\theta_{v,t})(w)e^{\theta_{v,t}}\omega_{t}^n
	\\
	\label{eq:fx}
	&=&
		\int_X d(h_{t}-\theta_{v,t})(w)e^{h_0-t\varphi_t}\omega_{0}^n.
\end{eqnarray}
Since
\[
	\mbox{Ric}(\omega_t)-L_{v}(\omega_t)=t\omega_t+(1-t)\omega_0,
\]
we have
\begin{equation}\label{eq:gap_krs_t}
	h_t-\theta_{v,t}=(t-1)\varphi_t.
\end{equation}
Combining (\ref{eq:fx}) and (\ref{eq:gap_krs_t}), we get
\[
	F_v(w)=(t-1)\int_X d\varphi_t(w)e^{h_0-t\varphi_t}\omega_{0}^n.
\]
From
\[
	L_w(e^{-t\varphi_t+h_0} \omega^n_0)
	=(L_w e^{-t\varphi_t})e^{h_{0}}\omega^n_0+e^{-t\varphi_t}L_w(e^{h_0}\omega^n_0),
\]
we have
\[
	0=\frac{t}{t-1}F_v(w)-\int_X e^{-t\varphi_t}L_w (e^h \omega^n_0).
\]
Since
\[
	e^{-t\varphi_t}L_w (e^h \omega^n_0)
	=
	e^{h-t\varphi_t}\mbox{div}_{e^h\omega^n}(w)\omega^n_0
	=
	\mbox{div}_{e^h\omega^n}(w)e^{\theta_{v, t}}\omega^n_t,
\]
we get
\begin{equation}\label{eq:fx_div}
	0=\frac{t}{t-1}F_v(w)-\int_X \mbox{div}_{e^h\omega^n}(w)e^{\theta_{v, t}}\omega^n_t.
\end{equation}
By (\ref{eq:vanishing_tianzhu_inv}) we have 
\begin{equation}\label{eq:fx_div2}
	0= \int_X \mbox{div}_{e^h\omega^n}(w)e^{\theta_{v, t}}\omega^n_t.
\end{equation}
Suppose $V_v \subset Z^+(w)$, and seek a contradiction.
Let $\delta > 0$ be the minimum of $\mbox{div}_{e^h\omega^n}(w)$ on $V_v$. Let $Y$ be the set of all points
in $X$ such that $\mbox{div}_{e^h\omega^n}(w) \le \frac{\delta}2$. Then $Y$ is a compact subset of $X - V_v$.
By Lemma \ref{lem:volume_zero} we have
\begin{equation}\label{eq:volume_zero2}
	\int_Y e^{\theta_{v.t}}\omega^n_t \to 0
\end{equation}
as $t \to t_{\infty}$.
From (\ref{eq:fx_div2}) and (\ref{eq:volume_zero2}) we see that
\begin{equation}\label{eq:volume_zero3}
	\int_{X-Y} \mbox{div}_{e^h\omega^n}(w)e^{\theta_{v.t}}\omega^n_t \to 0.
\end{equation}
Since $\mbox{div}_{e^h\omega^n}(w) \ge \frac{\delta}2$ we have 
\begin{equation}\label{eq:volume_zero4}
	\int_{X-Y} e^{\theta_{v.t}}\omega^n_t \to 0.
\end{equation}
Adding (\ref{eq:volume_zero2}) and (\ref{eq:volume_zero4}) we get
\begin{equation}\label{eq:volume_zero5}
	\int_{X} e^{\theta_{v.t}}\omega^n_t \to 0
\end{equation}
which contradicts (\ref{eq:def_theta}). This completes the proof of Theorem 1.4.
\end{proof}

%%%%% Additional part by Sano (11/10/07) %%%%%%

\section{Application of Theorem \ref{thm:nadel_thm_krs}}
In this section, we shall show that $\mathbb{CP}^2 \sharp  \overline{\mathbb{CP}^2}$ admits a 
K\"ahler-Ricci soliton. 
We keep the notations in the previous sections. 
Let $G'$ be as in the proof of Corollary \ref{thm:main_theorem2}. 
Assume that $X$ would not admit any K\"ahler-Ricci soliton. 
As in the proof of Corollary \ref{thm:main_theorem2}, the invariance of MIS under the symmetry of $X$ implies that the KRS-MIS could be either $E$ or the $(+1)$-curve.
Let $w$ be the holomorphic vector field induced by a vector $(1, 1)\in N_{\mathbb{R}}$. The induced flow on $N_{\mathbb{R}}$ preserves a line $\{(s, s+c)\in N_{\mathbb{R}}\mid s\in \mathbb{R}\}$ for all $c \in\mathbb{R}$. Since the relative interior of the facets $\overline{w_3w_4}$ and $\overline{w_1w_2}$ are equal to $\{\lim_{s\to-\infty}\mu_g((s, s+c)) \mid c\in \mathbb{R}\}$ and $\{\lim_{s\to\infty}\mu_g((s, s+c)) \mid c\in \mathbb{R}\}$ respectively,  the vector field $w$ induces the flow on $X$ which flows 
from $E$ towards the $(+1)$-curve while fixing them. From (\ref{eq:divpositive_condition}), we find that the divergence of $w$ is positive on the exceptional divisor. Hence, we have that the KRS-MIS is not $E$ due to Theorem \ref{thm:nadel_thm_krs}.
Similarly, since the divergence of $-w$ is positive on the $(+1)$-curve, the KRS-MIS is not the $(+1)$-curve. 
Therefore, we prove that $\mathbb{CP}^2 \sharp \overline{\mathbb{CP}^2}$ admits a K\"ahler-Ricci soliton.

\vspace{0.5cm}

\noindent
{\bf Acknowledgements} : 
The second author is supported by JSPS-EPDI fellowship. This work is done while the second author stayed in DPMMS, the University of Cambridge and IH\'ES, and he would like to thank for their hospitality. The second author would like to thank Alexei Kovalev for his kind support during the stay in DPMMS and Jacopo Stoppa.

\section{Appendix}
In this Appendix, we shall show that the MIS induced from the K\"ahler-Ricci flow on the surface obtained by blowing up of $\mathbb{C}\bfP^2$ at $p_1$ and $p_2$ equals to the tree of the $(-1)$-curves. 
In this case, the set of vertices of the associated polytope $\Delta$ consists of
\begin{center}
	$w_1=(1, 0)$, $w_2=(1, -1)$, $w_3=(-1, -1)$, $w_4=(-1, 1)$ and $w_5=(0, 1)$.
\end{center}
The set of the fundamental generators of one-dimensional cones in the associated fan $\Sigma$ consists of
\begin{center}
	$(1,0)$, $(0,-1)$, $(-1,0)$, $(0,1)$ and $(1,1)$.
\end{center}
The moment map
$
	\mu_{\tilde{u}}(x_1, x_2)
	=\mbox{grad}\tilde{u}
	:=
	\bigg(
		\frac{\partial \tilde{u}}{\partial x_1}(x), \frac{\partial \tilde{u}}{\partial x_2}(x)
	\bigg)
$ equals to
\begin{eqnarray}\label{moment2}
		\lefteqn{						
		e^{-\tilde{u}}
		\big(
		e^{x_1}+e^{x_1-x_2}-e^{-x_1-x_2}-e^{-x_1}-e^{-x_1+x_2},
			}
		\hspace{1cm}
	\\
	&&
	\nonumber
		-e^{x_1-x_2}-e^{-x_2}-e^{-x_1-x_2}+e^{-x_1+x_2}+e^{x_2}
	\big)	
\end{eqnarray}
where $e^{\tilde{u}}=e^{x_1}+e^{x_1-x_2}+e^{-x_2}+e^{-x_1-x_2}+e^{-x_1}+e^{-x_1+x_2}+e^{x_2}+1$.
Then $X$ has the $\mathbb{Z}_2$-symmetry, whose set of fixed points in $M_{\mathbb{R}}$ equals to $\{(y, y)\in M_{\mathbb{R}}\mid y\in \mathbb{R}\}$.
Let us introduce some notations. For a $G$-invariant K\"ahler potential $\varphi$ such that $\omega_{g}+\frac{\sqrt{-1}}{2\pi}\partial\bar{\partial}\varphi>0$ and $\sup \varphi=0$, let $\tilde{\varphi}$ be the associated $\mathbb{R}$-valued function on $N_{\mathbb{R}}$. (We can define $\tilde{\varphi}$, because $\varphi$ is 
$T_{\bfR}$-invariant, so $\varphi|_U$ descends to a function on $N_{\mathbb{R}}\simeq U/T_{\bfR}$.)
Tian and Zhu (\cite{tian-zhu07}) proved the theorem on convergence of the K\"ahler-Ricci flow on a compact K\"ahler manifold with a K\"ahler-Ricci soliton. After that, Zhu (\cite{zhu0703}) also proved the same result on a toric Fano manifold without the assumption of the existence of K\"ahler-Ricci solitons. Let us recall their results. 
Let $(X, g_0)$ be a Fano manifold where $g_0 \in c_1(X)$. The (normalized) K\"ahler-Ricci flow is defined by
\begin{equation}
	\label{eq:kahler_ricci_flow}
		\left\{\begin{array}{l}
		\frac{\partial g(t,\cdot)}{\partial t}
		=
		-\mbox{Ric}(g(t,\cdot)) +g(t,\cdot), 
		 \\
		 g(0, \cdot)
		 =g_0.
		 \end{array}\right. 
\end{equation}
\begin{theorem}[Tian-Zhu, \cite{tian-zhu07}]
	\label{thm:tian_zhu_ricciflow}
Let $X$ be a compact K\"ahler manifold which admits a K\"ahler-Ricci soliton $(g_{KS}, v)$.
Then any solution $g(t,\cdot)$ of (\ref{eq:kahler_ricci_flow}) will converge to the K\"ahler-Ricci soliton 
$g_{KS}$ in the sense of Cheeger-Gromov if the initial K\"ahler metric $g_0$ is $K_v$-invariant, where $K_v$ is the one-parameter subgroup of $K$ generated by the imaginary part of $v$.
\end{theorem}
By using the above theorem, we can calculate the MIS coming from the K\"ahler-Ricci flow on $\mathbb{CP}^2\sharp 2\overline{\mathbb{CP}^2}$. 
More precisely, we shall determine the $G^{\mathbb{C}}$-invariant MIS on $\mathbb{CP}^2\sharp 2\overline{\mathbb{CP}^2}$ induced by the sequence $\{\varphi_{t_i}\in C^{\infty}(X)\mid \sup \varphi_{t_i}=0\}$ of the K\"ahler potentials such that $g(t_i,\cdot):=g_0(\cdot)+(\partial_a\bar{\partial}_b\varphi_{t_i}(\cdot))_{ab}$ is a solution of (\ref{eq:kahler_ricci_flow}) at $t=t_i$, where $G$ is defined in the section \ref{sec:main_thm}.
Remark that the existence of K\"ahler-Ricci solitons on $\mathbb{CP}^2\sharp 2\overline{\mathbb{CP}^2}$ is assured by the result of Wang and Zhu (\cite{wang-zhu}) and also Zhu's result (\cite{zhu0703}) allows us to apply Theorem \ref{thm:tian_zhu_ricciflow} to $\mathbb{CP}^2\sharp 2\overline{\mathbb{CP}^2}$ without the assumption of the existence of K\"ahler-Ricci solitons. 
On $\mathbb{CP}^2\sharp 2\overline{\mathbb{CP}^2}$, let $g_0$ be the $G$-invariant metric defined by (\ref{eq:invariant_metric}).
Let $v$ be the holomorphic vector field on $\mathbb{CP}^2\sharp 2\overline{\mathbb{CP}^2}$ satisfying (\ref{eq:vanishing_tianzhu_inv}) which is invariant under the $\mathbb{Z}_2$-action and associated to the K\"ahler-Ricci soliton $g_{KS}$. Let $\sigma_t=\exp(tv)$. Since $g_0$ and $v$ are $\mathbb{Z}_2$-invariant, the arguments in \cite{tian-zhu07} and \cite{zhu0703} imply that there is a sequence $\{t_i\}_i$ such that $\omega'_{t_i}:=(\sigma_{t_i})^*\omega_{t_i}$ converges to $\omega_{KS}$, where $\omega_{t_i}$ and $\omega_{KS}$ are K\"ahler forms of $g(t_i,\cdot)$ and $g_{KS}$ respectively. Let $\varphi'_{t_i}$ and $\psi_{t_i}$ be K\"ahler potential functions satisfying
\[
	\omega'_{t_i}=\omega_0+\frac{\sqrt{-1}}{2\pi}\partial\bar{\partial}\varphi'_{t_i},
	\,\,\,
	\sup \varphi'_{t_i}=0,
\]
and
\[
	(\sigma_{t_i}^{-1})^*\omega_0=\omega_0+\frac{\sqrt{-1}}{2\pi}\partial\bar{\partial}\psi_{t_i},
	\,\,\,
	\sup \psi_{t_i}=0
\]
respectively, where $\omega_0$ is the K\"ahler form of $g_0$. 
Then we find that $\varphi_{t_i}$ equals to $\psi_{t_i}+(\sigma_{t_i}^{-1})^*\varphi'_{t_i}-\sup(\psi_{t_i}+(\sigma_{t_i}^{-1})^*\varphi'_{t_i})$. 
Since $\omega'_{t_i}$ converges to $\omega_{KS}$, we find that $\varphi'_{t_i}$ also converges a smooth function on $X$. This implies $\|(\sigma_{t_i}^{-1})^*\varphi'_{t_i}\|_{C^0}\le C$, where $C$ is a constant independent of $t$. So we have
\begin{equation}\label{eq:bdd_psi_varphi'}
	\|(\sigma_{t_i}^{-1})^*\varphi'_{t_i}-\sup(\psi_{t_i}+(\sigma_{t_i}^{-1})^*\varphi'_{t_i})\|_{C^0}\le C.
\end{equation}
Therefore from its definition and (\ref{eq:bdd_psi_varphi'}) we find that the multiplier ideal sheaf induced from $\{\varphi_{t_i}\}$ equals to the one induced from $\{\psi_{t_i}\}$.
From the $\mathbb{Z}_2$-symmetry of $\mathbb{CP}^2\sharp 2\overline{\mathbb{CP}^2}$, we can determine the vector field $v$ associated to the K\"ahler-Ricci soliton by the sign of the (ordinary) Futaki invariant $F(v)$. In general, we have the following lemma.
\begin{lemma}
	\label{lem:riccisoliton_futakiinvariant}
Let $(X, g_{KS}, v)$ be a compact Fano manifold with K\"ahler-Ricci soliton $(g_{KS}, v)$. Suppose that $v$ satisfies (\ref{eq:vanishing_tianzhu_inv}). Then the ordinary Futaki invariant $F(v)$ is positve.
\end{lemma}
\begin{proof}
Let $\theta_v(g)$ be a real-valued smooth function satisfying
$\iota_v(\omega_g)=\frac{\sqrt{-1}}{2\pi}\bar{\partial}\theta_v(g)$
which is determined uniquely up to constant. The fact that $\theta_v(g)$ is real-valued follows from that the imaginary part of $v$ is a Killing vector field.
Since $v$ is associated with the K\"ahler-Ricci soliton, we have
\[
	\mbox{Ric}(\omega_{KS})-\omega_{KS}
	=
	\frac{\sqrt{-1}}{2\pi}\partial\bar{\partial}h_{g_{KS}}
	=
	\frac{\sqrt{-1}}{2\pi}\partial\bar{\partial}\theta_v(g),
\]
hence we find that $h_{g_{KS}}$ equals to $\theta_v(g)$ up to constant.
From the definition of the ordinary Futaki invariant, we get
\begin{eqnarray*}
		F(v)
	&=&
		\int_X d\theta_v(g) (v) \omega^n_{g_{KS}}
	=
		\int_X |\bar{\partial} \theta_v(g)|^2 \omega^n_{g_{KS}}>0.
\end{eqnarray*}
\end{proof}
Lemma \ref{lem:riccisoliton_futakiinvariant} implies that $v$ equals to the holomorphic vector field induced by $\beta(1, 1)\in N_{\mathbb{R}}$ for some positive constant $\beta$. (The one-parameter subgroup $\sigma_t$ flows from $p_0$ towards the proper transform of the line between $p_1$ and $p_2$.) In fact, we can show that the invariant $F(v)$ is positive by using the localization formula for the invariant $F$ (cf. \cite{futakimorita85}, \cite{futaki}).
Let $\tilde{\psi}_{t_i}$ be the real-valued function on $N_{\mathbb{R}}$ corresponding to $\psi_{t_i}$.
Since $\frac{\sqrt{-1}}{2\pi}\partial\bar{\partial}(u+\psi_{t_i})=(\sigma_{t_i}^{-1})^{*}\omega_0$ on $U$ and $\sup_{N_{\mathbb{R}}}\tilde{\psi}_{t_i}=0$, so we have
\begin{eqnarray*}
		(\tilde{u}+\tilde{\psi}_{t_i})(x_1, x_2)
	&=&
		\tilde{u}(x_1-\beta t_i, x_2-\beta t_i)-2\beta t_i
	\\
	&=&
		\log
		(e^{x_1-3\beta t_i}+e^{x_2-3\beta t_i}+e^{-x_1+x_2-2\beta t_i}+
		1+e^{x_1-x_2-2\beta t_i}+
	\\
	&&
		e^{-x_1-\beta t_i}+e^{-x_2-\beta t_i}+e^{-x_1-x_2}).	
\end{eqnarray*}
In fact, it is easy to check that $\sup_{N_{\mathbb{R}}}\tilde{\psi}_{t_i} \le 0$ and
\begin{eqnarray*}
	\tilde{\psi}_{t_i}(-s,-s)
	=
	\log
	\biggl(
	\frac
		{2e^{-s-3\beta t_i}+3e^{-2\beta t_i}+2e^{s-\beta t_i}+e^{2s}}
		{2e^{-s}+3+2e^{s}+e^{2s}}
	\biggr)
	\to
	0
\end{eqnarray*}
as $s\to\infty$ for all $t_i$.
Now let us see that the $(-1)$-curve represented by the segment $\{(y_1, 1) \in M_{\mathbb{R}} \mid -1\le y_1\le 0\}$ is contained in the MIS induced by $\{\psi_{t_i}\}$.
For this purpose, it is sufficient to show that for \textit{any} open neighborhood $W$ of some point contained  in the above $(-1)$-curve such that the integral $\int_W \exp (-\alpha\psi_{t_i})\omega_0^{n}$ diverges to $\infty$ as $t_i\to \infty$ for \textit{all} $\alpha\in (n/(n+1), 1)$.
Let $p\in X$ be a point such that $\mu_{g_0}(p)=(-1/2, 1)\in M_{\mathbb{R}}$ and fix an open neighborhood $W_p\subset X$ of $p$.
On the half-line $\{(0,s) \in N_{\mathbb{R}}\mid s\ge 0\}$,
\begin{eqnarray}
	\nonumber
		(\tilde{u}+\tilde{\psi}_{t_i})(0,s)
	&=&
		\log
		(e^{-3\beta t_i}+e^{s-3\beta t_i}+e^{s-2\beta t_i}+
	\\
	\nonumber
	&&
	\qquad	e^{-2\beta t_i}+e^{-s-2\beta t_i}+e^{-\beta t_i}+e^{-s-\beta t_i}+e^{-s})\\
	&=&
		\log e^{-s}(e^{s-3\beta t_i}+e^{2s-3\beta t_i}+e^{2s-2\beta t_i}+
	\\
	\nonumber
	&&
	\qquad	e^{s-2\beta t_i}+e^{-2\beta t_i}+e^{s-\beta t_i}+e^{-\beta t_i}+1).
\end{eqnarray}
So we have
\begin{equation}
	\label{eq:u'_estimate}
		(\tilde{u}+\tilde{\psi}_{t_i})(0,s) \le -s+C
\end{equation}
for all $0 \le s \le \beta t_i$, where $C$ is independent of $t_i$ and $s$.
On the other hand, we have
\begin{equation}
	\label{eq:u_estimate}
		\tilde{u}(0, s) \le s+C 
\end{equation}
for all $s \in \mathbb{R}_{\ge 0}$.
Let $\tilde{W}_{+, \epsilon, s_0}=\{(x_1, x_2)\in N_{\mathbb{R}}; |x_1|<\epsilon, \,\, x_2\ge s_0\}$ for positive constants $\epsilon>0$ and $s_0>0$.
Then we can find that the closure of $\mu_{\tilde{u}}(\tilde{W}_{+, \epsilon, s_0})$ in $M_{\mathbb{R}}$ is contained in the image of $W_p$ under $\mu_{g_{0}}$ for sufficiently small $\epsilon$ and sufficient large $s_0$, since $\lim_{s\to\infty}\mu_{\tilde{u}}(0,s)=(-\frac{1}{2}, 1)$.
Then it is sufficient to estimate the integral of $\exp(-(\tilde{u}-\alpha \tilde{\psi}_{t_i}))$ over the region $W_{+, \epsilon, s_0}$, because
\begin{equation}\label{eq:equivalent_volumeform}
	\int_{W_p}\exp(-\alpha\psi_{t_i})\omega_{0}^n
	\ge
	C \int_{\tilde{W}_{+, \epsilon, s_0}} 
	\exp(-\tilde{u}-\alpha \tilde{\psi}_{t_i}) dx_1dx_2.
\end{equation}
The inequality (\ref{eq:equivalent_volumeform}) follows from that there are constants $\tilde{c}, \tilde{C}>0$ such that
\[
	\tilde{c} \le
	e^{\tilde{u}}\det \biggl(
	\frac{\partial^2 \tilde{u}}{\partial x_i \partial x_j}
	\biggr)
	\le 
	\tilde{C}.
\]
(See Lemma 4.3 in \cite{song} and \cite{batyrev-selivanova}.)
Since there is a constant $C_{\epsilon}$ depending only on $\epsilon$ such that
\begin{equation}\label{eq:bdd_on_W}
	|(\tilde{u}+\tilde{\psi}_{t_i})(0, x_2) -(\tilde{u}+\tilde{\psi}_{t_i})(x_1, x_2)| \le C_{\epsilon}
\end{equation}
where $(x_1, x_2)\in \tilde{W}_{+,\epsilon, s_0}$,
from (\ref{eq:u'_estimate}), (\ref{eq:u_estimate}) and (\ref{eq:bdd_on_W})
we find that for any $\alpha>1/2$ and sufficient large $t_i$
\begin{eqnarray}
	&&
	\nonumber
		\int_{\tilde{W}_{+,\epsilon, s_0}} 
		\exp(-\tilde{u}-\alpha \tilde{\psi}_{t_i}) dx_1dx_2
	\\
	\nonumber
	&\ge&
		C_\epsilon \int_{s_0}^{\beta t_i} 
		\exp\{-(1-\alpha)\tilde{u}(0,s)-\alpha (\tilde{u}+\tilde{\psi}_{t_i})(0,s)\} ds
	\\
	\nonumber
	&\ge&
		C_\epsilon \int_{s_0}^{\beta t_i}
		\exp((2\alpha-1)s)ds
	\\
	\label{eq:integral_W}
	&=&	
		\frac{C_\epsilon}{(2\alpha-1)}(e^{(2\alpha-1)\beta t_i}-e^{(2\alpha-1)s_0})
		\to \infty	
\end{eqnarray}
as $t_i \to \infty$.
The inequality (\ref{eq:integral_W}) implies that the $(-1)$-curves intersecting with $(0)$-curves are contained in the MIS, because the desired MIS is $\mathbb{Z}_2$-symmetric.
On the other hand, let us see that the $(0)$-curve represented by the segment $\{(y_1, -1) \in M_{\mathbb{R}} \mid -1<y_1<1\}$ is not contained in the MIS.
For this purpose, it is sufficient to show that for \textit{some} open neighborhood $W$ of some point contained  in the above $(0)$-curve such that the integral $\int_W \exp (-\alpha\psi_{t_i})\omega_0^{n}\le \mathcal{O}(1)$  as $t_i\to \infty$ for \textit{some} $\alpha\in (n/(n+1), 1)$.
Let $W_{-, \epsilon, 0}=\{(x_1, x_2)\in N_{\mathbb{R}}; |x_1|<\epsilon, \,\, x_2\le 0\}$ for sufficiently small $\epsilon>0$.
Since $\lim_{s\to -\infty}\mu_{\tilde{u}}(0,s)=(0,-1)\in M_{\mathbb{R}}$, there is an open subset $W\subset X$ such that $(0,-1) \in \overline{\mu_{g_0}(W)}\subset \mu_{\tilde{u}}(W_{-,\epsilon, 0})$, where $\overline{\mu_{g_0}(W)}$ is the closure of $\mu_{g_0}(W)$.
On the half-line $\{(0,s) \in N_{\mathbb{R}}\mid s\le 0\}$, we have  $(\tilde{u}+\tilde{\psi}_{t_i})(0,s) \ge 0$ for all $s\le 0$. This implies that for any $\alpha\in (0,1)$
\begin{eqnarray*}
		\int_W \exp(-\alpha\psi_{t_i})\omega_0^n
	&\le&
		C \int_{W_{-,\epsilon, 0}}
		\exp(-\tilde{u}-\alpha \tilde{\psi}_{t_i}) dx_1dx_2
	\\
	&\le&
		C_\epsilon  \int_{-\infty}^{0} 
		\exp\{-(1-\alpha)\tilde{u}(0,s)-\alpha (\tilde{u}+\tilde{\psi}_{t_i})(0,s)\} ds
	\\
	&\le&
		C_\epsilon  \int_{-\infty}^{0} 
		\exp(-(1-\alpha)\tilde{u}(0,s)) ds
	\\
	&\le&
		C_\epsilon \int_{-\infty}^{0}
		\exp((1-\alpha)s)ds
	\\
	&\le&
		C_\epsilon.
\end{eqnarray*}
In the above inequalities, we used $\tilde{u}(0,s)\ge -s$ for all $s\le 0$.
This implies that the two $(0)$-curves are not contained in the MIS.
Since the Nadel's vanishing theorem (\ref{eq:vanishing_thm}) induces that the multiplier ideal subschemes are connected (see \cite{nadel1}), so we find that the $G^{\mathbb{C}}$-invariant MIS induced by $\{(\sigma_{t_i}^{-1})^*\omega_0\}$ equals exactly to the tree of the three $(-1)$-curves in $\mathbb{CP}^2\sharp 2\overline{\mathbb{CP}^2}$.
By the similar calculation, we also find that the $G^{\mathbb{C}}$-invariant MIS coming from the K\"ahler-Ricci flow on $\mathbb{CP}^2\sharp \overline{\mathbb{CP}^2}$ equals to the exceptional divisor.

%%%%%%%%%%%% END by Sano %%%%%%%%%%%%%%%%%

\end{document}